\documentclass[letterpaper, 10pt,conference]{ieeeconf}
\usepackage{blindtext}
\usepackage{graphicx}
\IEEEoverridecommandlockouts   
\usepackage{amssymb,amsmath,amsfonts,color}
\usepackage{amsmath,amsfonts}
\usepackage{tikz,cite}
\usepackage{epstopdf} 
\usepackage{psfrag}
\usepackage{latexsym}
\usepackage{epstopdf}
\usepackage{xspace}

\usetikzlibrary{decorations.pathmorphing}

\newtheorem{theorem}{Theorem}

\newtheorem{proposition}{Proposition}
\newtheorem{lemma}{Lemma}
\newtheorem{remark}{Remark}
\newtheorem{definition}{Definition}
\newtheorem{example}{Example}
\DeclareMathOperator{\diag}{diag}

\usepackage{tcolorbox}
\usepackage{algorithm}

\usepackage{soul}

\renewcommand{\baselinestretch}{1.0001}

\begin{document}
\title{\LARGE \bf  Resilient Estimation and Control on $k$-Nearest Neighbor Platoons:  A Network-Theoretic Approach}

\author{Mohammad Pirani, Ehsan Hashemi, Baris Fidan, John W. Simpson-Porco, Henrik Sandberg,\\ Karl Henrik Johansson
\thanks{ M. Pirani, H. Sandberg and K. H. Johansson are with the department of automatic control, KTH Royal Institute of Technology, Stockholm, Sweden  E-mail: {\texttt{pirani, hsan, kallej@kth.se}}.
E. Hashemi and B. Fidan   are with the Department of Mechanical and Mechatronics Engineering, University of Waterloo, ON, N2L 3G1, Canada  (email: ehashemi,  fidan @uwaterloo.ca).  John W. Simpson-Porco is with the Department of Electrical and Computer Engineering at the University of Waterloo, Waterloo, ON, Canada. E-mail: {\texttt{jwsimpson@uwaterloo.ca}}.
}}

\maketitle

\thispagestyle{empty}
\pagestyle{empty}


\begin{abstract}
This paper is concerned with the network-theoretic properties of so-called $k$-nearest neighbor intelligent vehicular platoons, where each vehicle communicates with $k$ vehicles, both in front and behind. The network-theoretic properties analyzed in this paper play major roles in quantifying the resilience and robustness of three generic distributed estimation and control algorithms against communication failures and disturbances, namely resilient distributed estimation, resilient distributed consensus, and robust network formation. Based on the results for the connectivity measures of the $k$-nearest neighbor platoon, we show that extending the traditional platooning topologies (which were based on interacting with nearest neighbors) to $k$-nearest neighbor platoons increases the resilience of distributed estimation and control algorithms to both communication failures and disturbances. Finally, we discuss how the performance of each algorithm scales with the size of the vehicle platoon.

\end{abstract}

\section{Introduction}

Intelligent transportation systems are an important real-world instance of a multi-disciplinary cyber-physical system \cite{Stankovic, Shensherman}. In addition to classical electromechanical engineering, designing intelligent transportation systems requires synergy with and between outside disciplines, including communications, control, and network theory. In this direction, estimation and control theory are pivotal parts in designing algorithms for the active safety of automotive and intelligent transportation systems \cite{Hedrikburelli, Trimpe, BartBesselink, PiraniITS}.  From another perspective, networks of connected vehicles are quite naturally mathematically modeled using tools from networks and  graph theory, with associated notions such as degree, connectivity and expansion. While these modeling tools are in general distinct, the primary goal of this paper is to investigate connections between the control-theoretic and network-theoretic approaches to intelligent platoons.

The interplay between the network and system-theoretic concepts in network control systems has attracted much attention in recent years \cite{Leonard, ArxiveRobutness}. There is a vast literature on revisiting the traditional system-theoretic notions from the network's perspective. In this direction, some new notions have emerged such as {\it network coherence} \cite{bamieh2012coherence, piranicdc} which is interpreted as the $\mathcal{H}_2$ and $\mathcal{H}_{\infty}$ norms of a network dynamical system showing the ability of the network in mitigating the effect of disturbances. Moreover, some other system performance metrics such as the controllability and observability have been revisited in networks \cite{Pasqualetti, Sundaram2011}. In all of these problems, network properties come into play in the form of necessary and/or sufficient conditions to satisfy specific system performance characteristics. The advantage of this approach is in large-scale networks for which working with systemic notions is a burdensome task and tuning network properties is more implementable.

The above-mentioned reciprocity between the system and network-theoretic concepts finds many applications in mobile networks and in particular in networks of connected vehicles. There is much research on designing distributed estimation and control algorithms for traffic networks to ensure the safety or optimality of the energy consumption \cite{BartBesselink, BartBesselink2, LiangJohansson}. In all of those settings, there exist system-theoretic conditions which ensure the effectiveness of the proposed algorithms. However, as the scale of the network increases and the interactions become more sophisticated, e.g., from simple platooning to more complex topologies, testing those system-theoretic conditions becomes harder and the need to redefine those conditions in terms of network-theoretic properties is seriously felt. To this end, our approach is to reinterpret the performance of distributed estimation and control algorithms in terms of graph-theoretic properties of $k$-nearest neighbor platoons. We first quantify how densely connected this network is, as there are many non-equivalent metrics used in the literature to quantify the network connectivity. Then we  make a connection between each connectivity measure with its corresponding system performance metric. From this view, the contributions of this paper are:

\begin{itemize}
    \item We discuss some network connectivity measures for a generalized form of vehicle platoons (called $k$-nearest neighbor platoon) and show that this particular network topology provides high levels of connectivity for most of the connectivity measures. Interestingly, most of these measures depend only on the number of local interactions of each vehicle in the platoon.

    \item We apply the connectivity measures of $k$-nearest neighbor platoon to provide network-theoretic conditions for the performance of three well-known distributed estimation and control algorithms and show the positive effect of such network topology in enhancing the resilience of those algorithms. We also discuss the role of the network scaling on the performance of each algorithm.

\end{itemize}
\textbf{Applications to Connected Vehicles:} The specific network structure discussed in this paper provides an appropriate representation of highway traffic networks. More particularly, the widely used Dedicated Short-Range Communications (DSRC), which are two-way short-to-medium range wireless communications, provide a communication channel which enables vehicles to communicate up to a specific distance (about 1000 meters) \cite{Harvey}. Hence, such a geometric-based topology in highway driving is compatible with the network structure discussed in this paper. There are different quantities that vehicles can share between each other via DSRC. Among them, some physical vehicle states, e.g., velocity or acceleration, important spatiotemporal parameters, such as road friction coefficients, or even vehicle's status, e.g., braking status, can be disseminated throughout the network based on today inter-vehicular communication standards. Each vehicles can use such information obtained from other vehicles to increase the reliability of its own estimation or measurement or to make predictions about the particular quantity that it will measure in the future.

The paper is organized as follows. After introducing some notations and definitions in Section \ref{sec:deffin},  we investigate network-theoretic connectivity measures of $k$-nearest neighbor platoons in Section \ref{sec:nettheprop}, namely {\it network vertex and edge connectivity}, {\it network robustness}, {\it network expansion} and {\it algebraic connectivity}. Based on those connectivity measures, in Section \ref{sec:disest} we discuss three distributed estimation and control algorithms and apply those network-theoretic results to the robustness measures of these algorithms. In particular, in Section \ref{sec:disest}. A, we discuss a {\it robust distributed estimation} technique on $k$-nearest neighbor platoons and a sufficient condition under which a vehicle can estimate states of other vehicles in a distributed manner, despite failures in inter-vehicular communication. In Section \ref{sec:disest}. B, vehicles try to reach a {\it consensus} on a value, e.g., velocity or road condition, and the algorithm should be again robust to communication failures between vehicles in the network. Finally, in Section \ref{sec:disest}. C, vehicles aim to perform a {\it robust network formation} algorithm (forming a stable, rigid platoon, with specified inter-vehicular distances), in the presence of communication disturbances. After introducing all three distributed estimation and control algorithms, in Theorems \ref{thm:robdeisdst}, \ref{thm:shreyashghgsmmaster} and \ref{thm:sqrt23}, we re-interpret these results in terms of specific network properties of $k$-Nearest Neighbor platoons and verify the results with some simulations. These analyses are schematically shown in Fig. \ref{fig:sstdvqwefnj}.
\begin{figure}[t!]
\centering
\includegraphics[scale=.42]{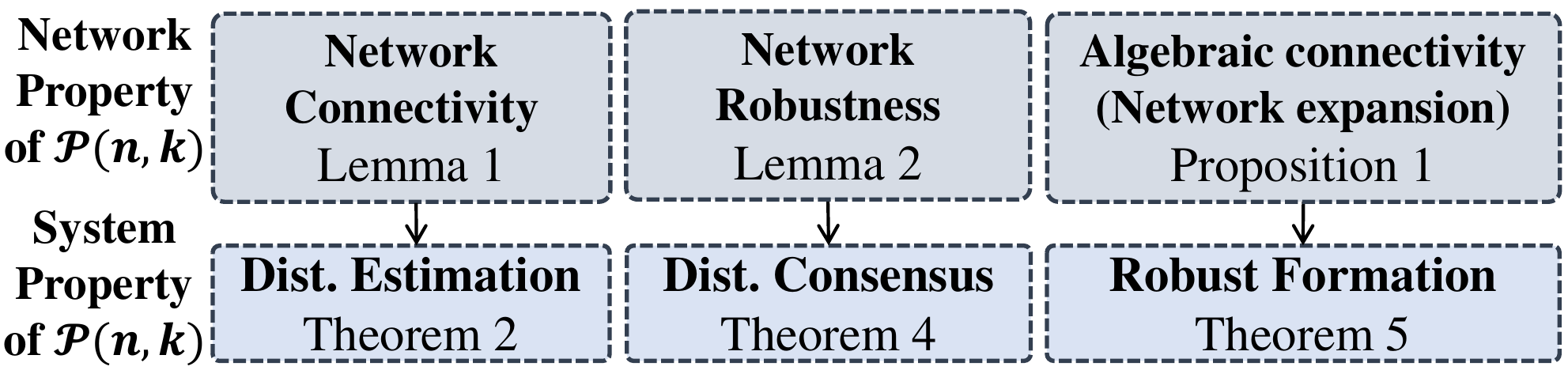}
\caption{Network-theoretic approaches to the performance of distributed estimation  and control algorithms on $k$-nearest neighbor platoons, $\mathcal{P}(n,k)$.}
\label{fig:sstdvqwefnj}
\end{figure}

\section{Notations and Definitions} \label{sec:deffin}
In this paper, an undirected network (graph) is denoted by  $\mathcal{G}=(\mathcal{V},\mathcal{E})$,  where $\mathcal{V} = \{v_1, v_2, \ldots, v_n\}$ is the set of nodes (or vertices) and $\mathcal{E} \subset \mathcal{V}\times\mathcal{V}$ is the set of edges. Neighbors of node $v_i \in \mathcal{V}$ are given by the set $\mathcal{N}_i = \{v_j \in \mathcal{V} \mid (v_i, v_j) \in \mathcal{E}\}$. The  degree of each node $v_i$ is denoted by $d_i=|\mathcal{N}_i|$ and the minimum and maximum degrees in graph $\mathcal{G}$ are shown by $d_{\rm min}$ and $d_{\rm max}$, respectively. The  adjacency matrix of the graph is a symmetric and binary $n \times n$  matrix $A$, where element $A_{ij}=1$ if $(v_i, v_j) \in \mathcal{E}$ and zero otherwise.  For a given set of nodes $X \subset \mathcal{V}$, the {\it edge-boundary} (or just boundary) of the set is defined as $\partial{X} \triangleq \{(v_i,v_j) \in \mathcal{E} \mid v_i \in X, v_j \in \mathcal{V}\setminus{X}\}$. The {\it isoperimetric constant} of $\mathcal{G}$ is defined as \cite{ChungSpectral}
 \begin{equation}
 i(\mathcal{G})\triangleq \min_{S \subset V, |S| \le \frac{n}{2}}\frac{|\partial S|}{|S|}.
 \label{eqn:iso}
 \end{equation}
 where $\partial S$ is the  edge-boundary of a set of nodes $S \subset V$. The Laplacian matrix of the graph is $L \triangleq D - A$, where $D = \diag(d_1, d_2, \ldots, d_n)$.  The eigenvalues of the Laplacian are real and nonnegative, and are denoted by $0 = \lambda_1(L) \le \lambda_2(L) \le \ldots \le \lambda_n(L)$ and $\lambda_2(L)$ is called the algebraic connectivity of the network \cite{Godsil}.  Given a connected  graph $\mathcal{G}$, an orientation of the graph $\mathcal{G}$ is defined by assigning  a direction (arbitrarily) to each edge in $\mathcal{E}$. For graph $\mathcal{G}$ with $m$ edges, numbered as $e_1, e_2, ..., e_m$, its node-edge incidence matrix $\mathcal{B}(\mathcal{G})\in \mathbb{R}^{n\times m}$ is defined as \cite{Fiedler}
$$[\mathcal{B}(\mathcal{G})]_{kl}=
  \begin{cases}
    1       & \quad  \text{if node $k$ is the head of edge $l$},\\
   -1  & \quad \text{if node $k$ is the tail of edge $l$},\\
   0  & \quad \text{otherwise}.\\
  \end{cases}
  $$
The graph Laplacian satisfies $L=\mathcal{B}(\mathcal{G})\mathcal{B}(\mathcal{G})^{\sf T}$ \cite{Godsil}. 

For positive integers $n,k \geq 1$ such that $n \geq k$, a $k$-Nearest Neighbor platoon containing $n$ vehicles, which we denote as $\mathcal{P}(n,k)$, is a specific class of networks which captures the physical properties of wireless sensor networks  in vehicular platoons. It is a network comprised of $n$ nodes (or vehicles),  where each node can communicate with its $k$ nearest neighbors from its back and $k$ nearest neighbors from its front, for some $k \in \mathbb{N}$. This definition is compatible with wireless sensor networks,  due to the limited sensing and communication range for each  vehicle and the distance between the consecutive vehicles \cite{PiraniITS}. An example of such network topology is shown in Fig. \ref{fig:2nndef}.
\begin{figure}[t!]
\centering
\includegraphics[width=1\linewidth]{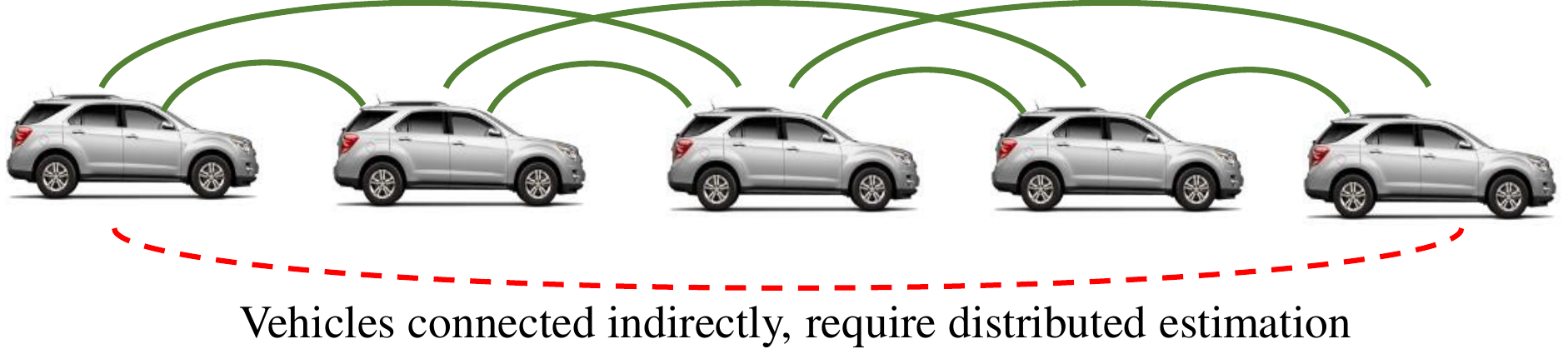}
\caption{An example of $\mathcal{P}(5,2)$ with $n=5$ and $k=2$. Green lines denote communication links.}
\label{fig:2nndef}
\end{figure}

\section{Network-Theoretic Properties} \label{sec:nettheprop}

 In this section, we examine four network connectivity measures which, as we will see, each play a fundamental role in understanding the system-theoretic performance of different algorithms on $k$-nearest neighbor platoons. These  properties, as mentioned in the previous sections, are {\it network connectivity, network robustness}, and {\it network expansion} and {\it algebraic connectivity}. Fig. \ref{fig:sstdvnj} (b) provides a visual sense of the strength of each of these connectivity measures in general graphs \cite{Ebiautomatica}. Fig \ref{fig:sstdvnj} (c) shows the values of each connectivity measure in $k$-nearest neighbor platoons which are discussed in detail in the subsequent subsections.  The main insight is that, while these connectivity notions are distinct in general networks, they collapse to one equivalent notion of connectivity for $k$-nearest neighbor platoons.

\subsection{Vertex and Edge Connectivity}
First, we have the following definitions of graph vertex and edge connectivities. 

\begin{definition}[\textbf{Cuts in Graphs}]
A \emph{vertex-cut} in a graph $\mathcal{G}=\{\mathcal{V}, \mathcal{E}\}$ is a subset $\mathcal{S} \subset \mathcal{V}$  of vertices such that removing the vertices in $\mathcal{S}$ (and any resulting dangling edges) from the graph causes the remaining graph to be disconnected. A \emph{$(j, i)$-cut} in a graph is a subset $\mathcal{S}_{ij}\subset \mathcal{V}$ such that if the nodes $\mathcal{S}_{ij}$ are removed, the resulting graph contains no path from vertex $v_j$ to vertex $v_i$.   Let $\kappa_{ij}$ denote the size of the smallest $(j, i)$-cut between any two vertices $v_j$ and $v_i$. The graph $\mathcal{G}$ is said to have \emph{vertex connectivity} $\kappa(\mathcal{G})=\kappa$ (or \emph{$\kappa$-vertex-connected}) if $\kappa_{ij}=\kappa$ for all $i,j \in \mathcal{V}$. The \emph{edge connectivity} $e(\mathcal{G})$ of a graph  $\mathcal{G}$ is the minimum number of edges whose deletion disconnects the graph.
\end{definition}

For the vertex and edge connectivity and graph's minimum degree the following inequalities hold 
\begin{equation}
\kappa(\mathcal{G})\leq e(\mathcal{G})\leq d_{\rm min}.
\label{eqn:veredgemin}
\end{equation}
\label{def:connectedge}
The following lemma discusses the connectivity of $k$-nearest neighbor platoons. 

\begin{lemma}
A $k$-nearest neighbor platoon $\mathcal{P}(n,k)$ is a $k$-vertex and a $k$-edge connected graph, i.e., $\kappa(\mathcal{G})=e(\mathcal{G})=k$.
\label{lem:kconnectedprop}
\end{lemma} 

\begin{proof}
We prove this result via contradiction. Suppose $\mathcal{P}(n,k)$ is a $\bar{k}$-connected graph, with $\bar{k}< k$. Thus, there exists a minimum vertex cut $\mathcal{S}_{ij}$ between two vertices $v_i$ and $v_j$ where $|\mathcal{S}_{ij}|=\bar{k}$. Without loss of generality, label the vertices from $v_i$ to $v_j$ as $v_i, v_{i+1}, ..., v_j$. Since $\bar{k}< k$, there is a vertex $\bar{v}$ among $v_{i+1},..., v_{i+k}$ (which are directly connected to $v_i$) which does not belong to $\mathcal{S}_{ij}$. By replacing $v_i$ with $\bar{v}$ in the above discussion, we will find a path from $v_i$ to $v_j$ which does not include vertices in $\mathcal{S}_{ij}$ and this contradicts the claim that $\mathcal{S}_{ij}$ is a vertex cut. Hence $\mathcal{P}(n,k)$ is a $k$-vertex connected graph.  For the edge connectivity, observe that for graphs $\mathcal{P}(n,k)$ we have $d_{\rm min}=k$.  The result then follows immediately from \eqref{eqn:veredgemin}.
\end{proof}

\subsection{Network Robustness}

The notion of network robustness is another network connectivity measure, which finds application in the study of distributed consensus algorithms.

\begin{definition}[ \textbf{$r$-Reachable/Robust Graphs}\cite{Leblank}]
Let $r \in \mathbb{N}$.  A subset $S \subset \mathcal{V}$ of nodes in the graph $\mathcal{G}=(\mathcal{V}, \mathcal{E})$ is said to be {\it $r$-reachable} if there exists a node $v_j \in S$ such that $|\mathcal{N}_j\setminus S| \ge r$.  A graph $G=(\mathcal{V}, \mathcal{E})$ is said to be {\it $r$-robust} if for every pair of nonempty, disjoint subsets of $\mathcal{V}$, at least one of them is $r$-reachable.
\label{def:robustnessr}
\end{definition}

Generally speaking, $r$-robustness is a stronger notion than $r$-connectivity \cite{Hogan}, as shown in the following example. 
\begin{example}
The graph shown in Fig. \ref{fig:sstdvnj} (a) is comprised of two complete graphs on $n$ nodes ($S_1$ and $S_2$) and each node in $S_1$ has exactly one neighbor in $S_2$ and vice-versa. The minimum degree and the vertex connectivity are both $n$; however, the network is only 1-robust. 
\begin{figure}[t!]
\centering
\includegraphics[scale=.42]{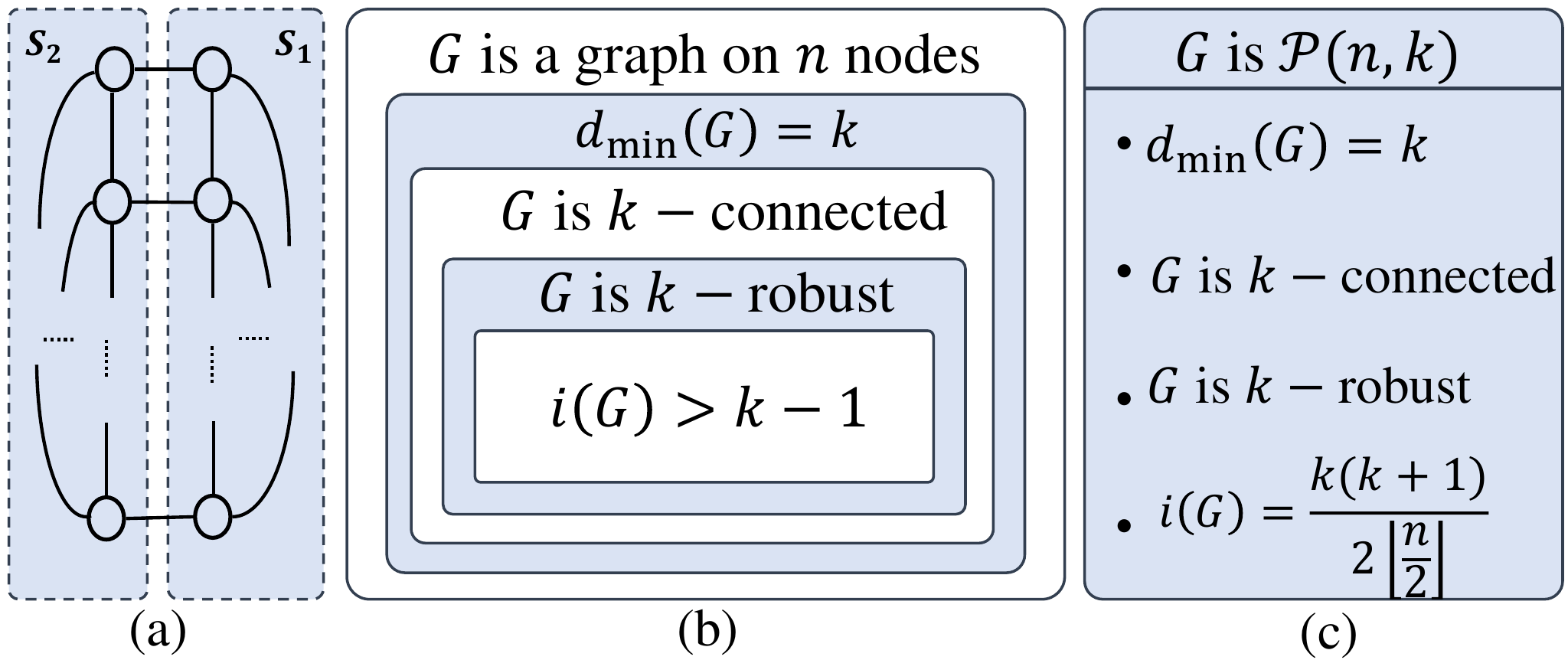}
\caption{(a) A graph with a large connectivity and small robustness, (b) Venn diagram of network connectivity measures for general graphs, (c) Connectivity measures for $k$-nearest neighbor platoons.}
\label{fig:sstdvnj}
\end{figure}
\label{exp:one}
\end{example}

As discussed in Example \ref{exp:one} and schematically shown in the Venn diagram in Fig. \ref{fig:sstdvnj} (b), the network minimum degree, network connectivity and network robustness have different strength in general graphs. However, our next result shows that these notions coincide for $k$-nearest neighbour platoons. 


 Based on the above definition of network robustness, we have the following lemma for the robustness of $k$-nearest neighbor platoons. 
 \begin{lemma}
 A $k$-nearest neighbor platoon $\mathcal{P}(n,k)$ is a $k$-robust network.
 \label{lem:rocbju}
 \end{lemma}
 \begin{proof}
  From Definition \ref{def:robustnessr}, by choosing every two disjoint sets of vertices in $\mathcal{P}(n,k)$, we see that the minimum number $r$ for which a subset is $r$-reachable is $r=d_{\rm min}$ and its corresponding subset is the ending node in the platoon, as shown in subset $A$ in Fig. \ref{fig:etdvnj}.
 \end{proof}

\begin{figure}[t!]
\centering
\includegraphics[scale=.45]{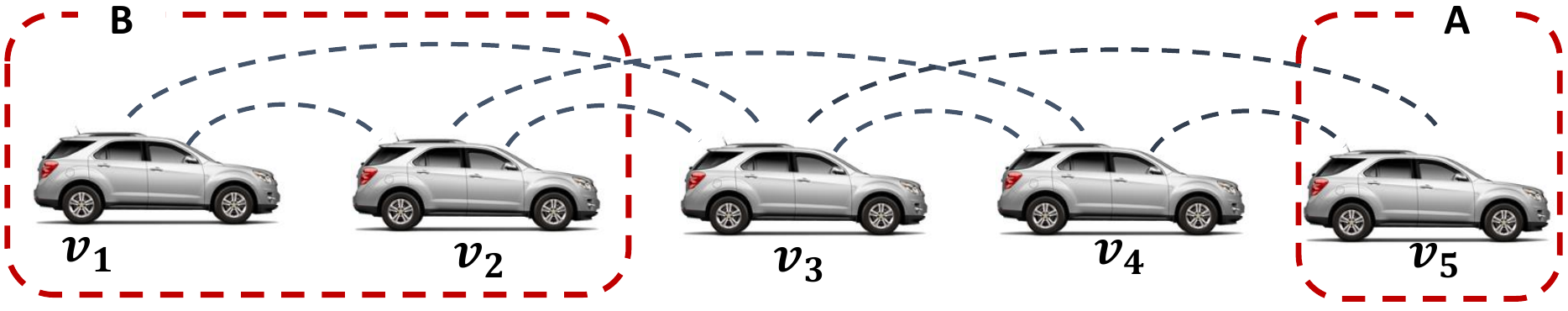}
\caption{Node Selection for calculating robustness (set $A$)  and isoperimetric constant (set $B$) in $\mathcal{P}(n,k)$.}
\label{fig:etdvnj}
\end{figure}

\subsection{Algebraic Connectivity and Network Expansion}

One can further extend inequalities \eqref{eqn:veredgemin} to include the algebraic connectivity and obtain \cite{Fiedler}
$$\lambda_2(L)\leq \kappa(\mathcal{G}) \leq e(\mathcal{G}) \leq d_{\rm min}(\mathcal{G}).$$
Based on the above inequalities, we conclude that the algebraic connectivity of $\mathcal{P}(n,k)$ is less than $k$. 
However, by using the notion of {\it network expansion}, some tighter bounds on the algebraic connectivity can be obtained.

\begin{definition}[\textbf{Expander Graph}]
Expander graphs are graph sequences for which each graph in the sequence has an expansion property, meaning that there exists $\gamma>0$ (independent of $n$) such that each subset S of nodes with size $|S| \leq  \frac{n}{2}$ has at least $\gamma|S|$ edges to the rest of the network. In particular, we say that the graph $\mathcal{G}$ is a $\gamma$-expander network if $i(\mathcal{G})= \gamma$ for some $\gamma>0$, where $i(\mathcal{G})$ is the isoperimetric constant defined in Section \ref{sec:deffin}. 
\end{definition}

Expander graphs have diverse applications in computer science and mathematics \cite{Hoory}. The algebraic connectivity of the graph is related to the network expansion (or the isoperimetric constant) by the following bounds \cite{ChungSpectral}
 \begin{equation}
\frac{i(\mathcal{G})^2}{2d_{\rm max}} \le \lambda_2(\mathcal{G})  \le 2 i(\mathcal{G}).
\label{eqn:lower_bound_lambda_2_iso}
\end{equation}
Using these bounds, we present the following proposition. 
 
\begin{proposition}
Given a $k$-nearest neighbor platoon $\mathcal{P}(n,k)$ its algebraic connectivity is bounded by
\begin{equation}\label{eqn:twolowerbounds}
\max\left\{2k-n+2, \frac{k(k+1)^2}{16\bar{n}^2}\right\} \leq \lambda_2(L) \leq \frac{2k(k+1)}{\bar{n}},
\end{equation}
where $\bar{n}=\lfloor \frac{n}{2} \rfloor$.
\label{prop:algconn}
\end{proposition}
 
\begin{proof}
First we use bounds given in \eqref{eqn:lower_bound_lambda_2_iso}. For this, we should calculate the isoperimetric constant in $\mathcal{P}(n,k)$ by finding a set in $\mathcal{P}(n,k)$ which minimizes $\frac{|\partial S|}{|S|}$ with $|S|\leq \frac{n}{2}$. A set which contains $\lfloor \frac{n}{2} \rfloor$ nodes, minimizes this function (Fig. \ref{fig:etdvnj}, set B). Hence, the isoperimetric constant will be $i(\mathcal{G})=\frac{1+2+...+k}{\lfloor \frac{n}{2} \rfloor}=\frac{k(k+1)}{2\lfloor \frac{n}{2} \rfloor}$. Substituting this value into \eqref{eqn:lower_bound_lambda_2_iso} and considering the fact that $d_{\rm max}\leq 2k$ provides the upper bound and the lower bound $\frac{k(k+1)^2}{16\bar{n}^2}$. The second lower bound comes from bound $2d_{\rm min}-n+2\leq \lambda_2(L)$ proposed in \cite{Fiedler} and considering the fact that $d_{\rm min}=k$. 
\end{proof}
The maximum over two lower bounds in \eqref{eqn:twolowerbounds} is due to the fact that for certain values of $k$ one of the lower bounds is tighter than the other. For instance, for $k\leq \frac{n-2}{2}$ the left lower bound is zero or negative and the right lower bound is tighter. However, for $k=n-1$ the left lower bound is tighter.

\begin{remark}[Comment on Mobile Networks]
For applications to wireless mobile networks, due to the mobility of agents in the network, the $\mathcal{P}(n,k)$ structure can change and some edges are added or removed. However, since all of the network connectivity measures discussed above are monotonic functions of the edge addition \cite{Ebiautomatica}, it is sufficient for the mobile network to preserve a minimum local connectivity $k$ to satisfy the all desired global connectivity measures. 
\end{remark}

\section{Distributed  Estimation and Control Algorithms}\label{sec:disest}

In this section, three estimation and control policies for vehicle platoons will be studied, and we will show how the connectivity measures introduced in Section \ref{sec:nettheprop} can be directly applied to quantify the performance of these algorithms. For each algorithm,  we will see that one of the network connectivity measures introduced in the previous section determines how much the algorithm is resilient to the effect of communication failures or disturbances.

\subsection{Distributed Estimation Robust to Communication Failures and Drops}

 Distributed estimation (or calculation) is a procedure by which vehicles in a network may estimate unavailable quantities based on incomplete localized measurements and cooperation with nearby vehicles.  Distributed estimation can potentially have diverse applications in vehicle networks, such as fault detection or prediction, as schematically shown in the upper box in Fig. \ref{fig:sstmgvnj}.
 
 \smallskip

The state of vehicle $v_j$, which can be its kinematic state, e.g., velocity, or some spatial parameter, e.g., road condition, is denoted simply by the scalar $x_j[0]$. The objective is to enable  vehicle $v_i$ in the network (which is not in the communication range of vehicle $v_j$) to calculate this value. To yield this, vehicle $v_i$ performs a linear iterative policy using the following  time invariant updating rule

\begin{equation}
x_i[k+1]=w_{ii}x_i[k]+\sum_{j\in \mathcal{N}_i}w_{ij}x_j[k]\,,
\label{eqn:1}
\end{equation}
where $w_{ii},w_{ij}>0$ are predefined weights. In addition to dynamics \eqref{eqn:1}, at each time step, vehicle $v_i$ has access to its own value (state) and the values of its neighbors. Hence, the vector of measurements for $v_i$ is defined as
\begin{equation}
y_i[k]=\mathcal{C}_i\boldsymbol{x}[k],
\label{eqn:measu}
\end{equation}
where $\mathcal{C}_i$ is a $(d_i +1)\times n$ matrix with a single 1 in each row that denotes the positions of the state-vector $\boldsymbol{x}[k]$ available to vehicle $v_i$ (i.e., these positions correspond to vehicles that are neighbors of $v_i$, along with vehicle  $v_i$ itself).

\begin{remark}\textbf{(Cyber-Physical Representation)}:
Fig. \ref{fig:sstmgvnj} provides a cyber-physical interpretation of the distributed estimation algorithm. According to this figure, algorithm \eqref{eqn:1} is developed in the cyber layer, which receives the physical states of vehicles from the physical layer as initial conditions for its algorithm (red dashed lines), perform the distributed estimation to obtain the initial states of all vehicles in the network, and finally returns those initial states  back to the physical layer (orange dashed lines). It should be noted that state $x_i[k]$ in \eqref{eqn:1} evolves in the cyber layer and it does not represent the evolution of vehicle's physical state based on the communication; the dynamics \eqref{eqn:1} is only used for implementing a distributed calculation algorithm. Here, it is only $\mathbf{x}[0]=[x_1[0], x_2[0], ..., x_n[0]]^T$ that reflects the physical states of the vehicles. 
   
\begin{figure}[t!] 
\centering
\includegraphics[scale=.42]{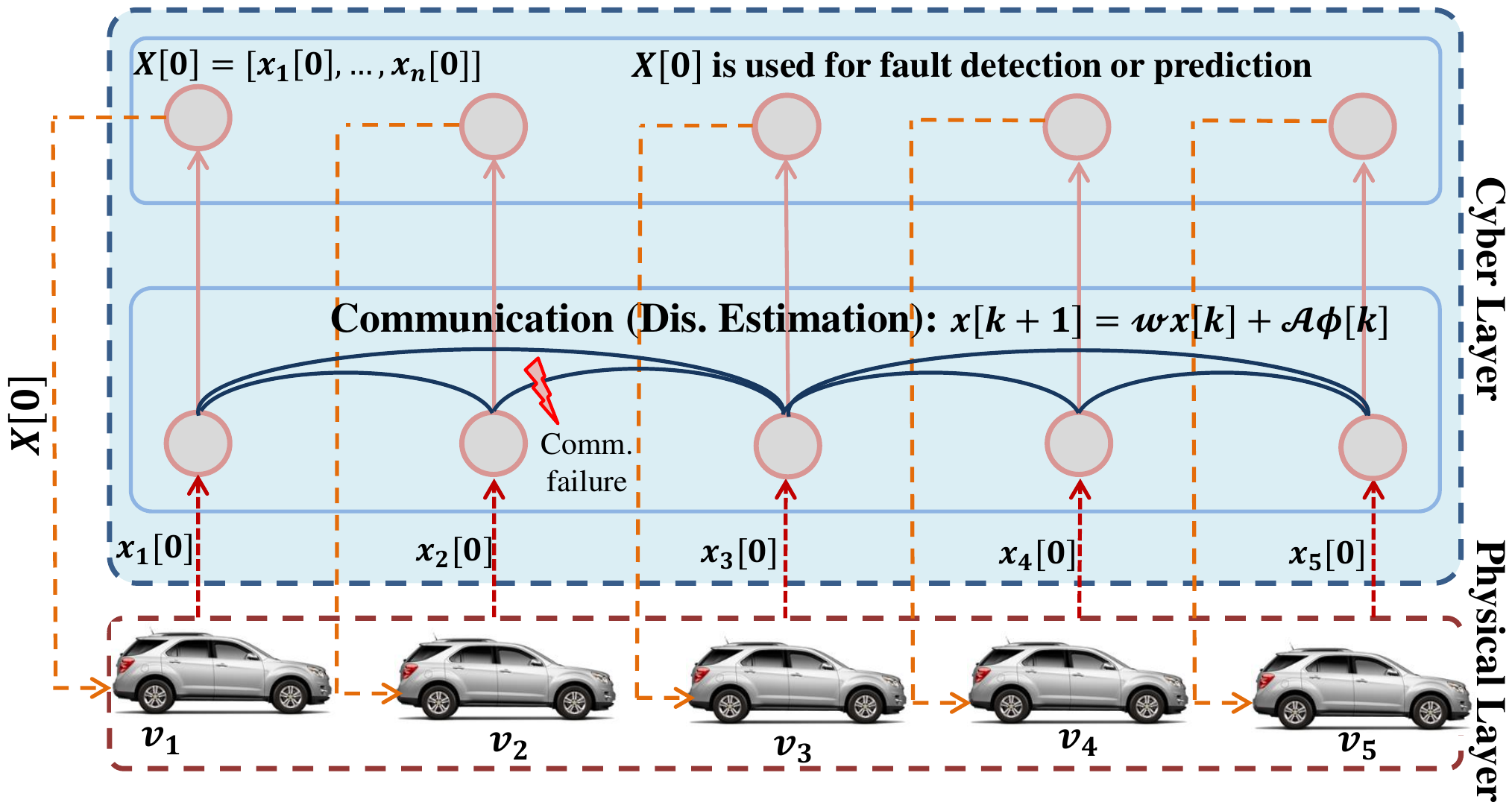}
\caption{Cyber-physical representation of the distributed estimation algorithm.}
\label{fig:sstmgvnj}
\end{figure}
\end{remark}

For such distributed estimation algorithms, we consider the possibility that there may exist some vehicles which fail to disseminate their information in a correct way, and some robust distributed estimation algorithms have been proposed to overcome such communication failures \cite{Sundaram2011, Sundaram}. More formally, suppose that some vehicles do not precisely follow \eqref{eqn:1} to update their value. In particular, at time step $k$, suppose vehicle $v_i$'s update rule deviates from the predefined policy \eqref{eqn:1} and (likely, unintentionally) adds an arbitrary value $\phi_i[k]$ to its updating policy.\footnote{In the literature such agents are called {\it adversarial} or {\it malicious} agents.} In this case, the updating rule \eqref{eqn:1} will become 
\begin{equation}
x_i[k+1]=w_{ii}x_i[k]+\sum_{j\in \mathcal{N}_i}w_{ij}x_j[k]+\phi_i[k],
\label{eqn:1wd}
\end{equation}
and if there are $f > 0$ of these faulty vehicles, \eqref{eqn:1wd} in vector form  becomes 
\begin{equation}
\boldsymbol{x}[k+1]=\mathcal{W} \boldsymbol{x}[k]+\underbrace{[\mathbf{e}_1 \hspace{3mm} \mathbf{e}_2 \hspace{3mm} ... \hspace{3mm} \mathbf{e}_f]}_{\mathcal{A}}\boldsymbol{\phi}[k],
\label{eqn:f2}
\end{equation}
where $\boldsymbol{x} = (x_1,\ldots,x_n)^{\sf T}$, $\mathcal{W} \in \mathbb{R}^{n\times n}$ is the matrix of communication weights $w_{ij}$, $\boldsymbol{\phi}[k]=[\phi_1[k], \phi_2[k], ..., \phi_f[k]]^{\sf T}$ and $\mathbf{e}_i$  denotes the $i$th unit vector of $\mathbb{R}^n$. The set of faulty vehicles in \eqref{eqn:f2} is unknown and consequently the matrix $\mathcal{A}$ is  unknown. However, each vehicle knows an upper bound for the number of faulty vehicles.

\begin{remark}[\textbf{Packet losses}]\label{Rem:PacketLoss}
The distributed calculation algorithm in the presence of vehicle communication fault analyzed in this paper contains the scenario which a vehicle stops receiving signal from its neighbors. This is the well-known notion called {\it  signal packet drop} which is studied in the communication literature \cite{Hajitouri, Sieler, Hespsurvey}. More formally, in \eqref{eqn:1wd} if we set $\phi_i[k]= -\sum_{j\in \mathcal{N}_i}w_{ij}x_j[k]$, it becomes equivalent to the case where $v_i$ does not receive the data from its neighbors. Since the analysis in this paper does not depend on the value of $\phi_i[k]$, the packet dropping scenario can be straightforwardly included in the robust distributed calculation analysis.
\end{remark}

The following theorem provides a condition which ensures that each vehicle is able to determine the (initial) states of all other vehicles in the network, despite of the action of some faulty vehicles. The details of the estimator design (which is in the form of an unknown input observer) is not discussed in this paper and we refer the reader to \cite{Sundaram2011}.

\begin{theorem}[\cite{Sundaram2011}]
Let  $\mathcal{G}(\mathcal{V},\mathcal{E})$ be a fixed graph and let $f$ denote the maximum number of faulty vehicles that are to be tolerated in the network. Then, regardless of the actions of the faulty vehicles, $v_i$ can uniquely determine all of the initial values of linear iterative strategy \eqref{eqn:f2} for almost\footnote{The {\it{almost}} in Theorem \ref{thm:robdeist} is due to the fact that  the set of  parameters for which the system is not observable has Lebesgue measure zero \cite{Reinschke}.} any choice of weights in the matrix $\mathcal{W}$  if $\mathcal{G}$ is at least  $(2f +1)$-vertex connected.
\label{thm:robdeist}
\end{theorem}

 Theorem \ref{thm:robdeist} provides a sufficient condition for each vehicle to be able to robustly distributedly estimate the initial states of other vehicles in the network. Theorem \ref{thm:robdeist} together with Lemma \ref{lem:kconnectedprop} yield the following theorem which shows the ability of $\mathcal{P}(n,k)$ in performing distributed estimation algorithms. 

\smallskip
 
\begin{theorem}
For a $k$-nearest neighbor platoon $\mathcal{P}(n,k)$, regardless of the actions of up to $\lfloor \frac{k-1}{2}\rfloor$ faulty vehicles, each vehicle can uniquely determine all of the initial values in the network via linear iterative strategy \eqref{eqn:f2} for almost any choice of weights in the matrix $\mathcal{W}$. 
\label{thm:robdeisdst}
\end{theorem}

\smallskip

 Fig. \ref{fig:sstss547vnj} illustrates via an example how Theorem \ref{thm:robdeisdst} provides network-theoretic sufficient condition for distributed estimation on $\mathcal{P}(n,k)$. In this example, there exists a single faulty vehicle in a network of 10 vehicles. Based on Theorem  \ref{thm:robdeisdst}, it is sufficient to have $\mathcal{P}(10,3)$ to overcome the action of the faulty vehicle; the corresponding trace in Fig \ref{fig:sstss547vnj} shows that the Euclidean norm of the  error of the estimated initial states of the vehicles in the network observed by a single vehicle goes to zero. More formally, if the true initial values are denoted by vector $\mathbf{x}[0]$ and the estimation (calculation) of these initial values by each vehicle at time step $k$ is $\hat{\mathbf{x}}[k]$, then Fig. \ref{fig:sstss547vnj} shows the Euclidean norm of the error vector $\mathbf{e}[k]=\hat{\mathbf{x}}[k]-\mathbf{x}[0]$. However, note that $\mathcal{P}(10,2)$ can also perform the distributed estimation algorithm well; this illustrates that the connectivity condition is sufficient, but not necessary.

\begin{figure}[t!]
\centering
\includegraphics[scale=.55]{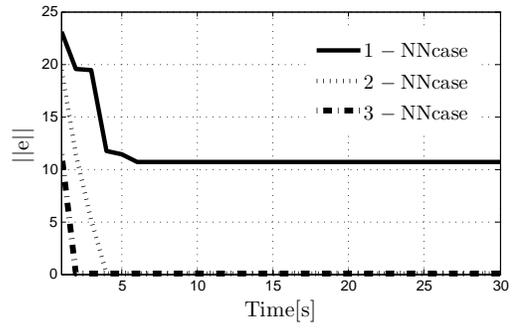}
\caption{Distributed estimation error for a vehicle in 1, 2 and 3 nearest-neighbour platoons of 10 vehicles.}
\label{fig:sstss547vnj}
\end{figure}

\subsection{Distributed Consensus, Robust to Communication Faults}

In the distributed consensus scenario, the network of connected vehicles tries to reach to a consensus value, e.g., velocity or road condition, despite the existence of some faults, biases, or signal drops in inter-vehicle communications.  In order to overcome the actions of faulty vehicles, the following iteration policy, called Weighted-Mean-Subsequence-Reduced (W-MSR) \cite{ShreyasHogan}, is proposed to overcome their actions.
\begin{definition}[\textbf{W-MSR Algorithm}\cite{ShreyasHogan}]
For some non-negative integer $f$, at each time-step, each node knows the number of faulty vehicles (or at least an upper bound of that) and disregards the largest and smallest $f$ values in its neighborhood ($2f$ in total) and updates its state to be a weighted average of the remaining values. More formally, this yields 
\begin{align}
{x}_{j}[k+1] = w_{jj}{x}_{j}[k]+ \sum_{p\in \mathcal{N}_j[k]} w_{jp}{x}_{p}[k].
\label{eqn:partiasdal}
\end{align}
where $\mathcal{N}_j[k]$ is the set of vehicles which are the neighbors of vehicle $j$ and are not ignored. 
\end{definition}

In particular, if there exist $f$ faulty vehicles, the dynamics is similar to \eqref{eqn:f2}, except the following two  additional restrictions on matrix $\mathcal{W}$:
\begin{enumerate}
    \item[(i)] $w_{jp}>0$, \quad  $\forall p\in \mathcal{N}_j[k]\cup \{v_j\}, v_j \in \mathcal{V}$,
     \item[(ii)]  $\sum_{p\in \mathcal{N}_j[k] \cup \{v_j\}} w_{jp}=1, \quad \forall v_j \in \mathcal{V}$.
\end{enumerate}
\begin{remark}
Unlike the distributed estimation algorithm that recovering $\boldsymbol{x}[0]$ was the final goal, in the distributed consensus, the final state of the communication dynamics, i.e., $\lim_{k\to \infty}x_i[k]$, is important as it determines the consensus value. Hence, our approach here is to design a  consensus algorithm which is resilient to faulty (or malicious) vehicles. 
\end{remark}

Similar to the case of distributed estimation mentioned in subsection A, the underlying network has to satisfy a certain level of connectivity to ensure that consensus can be achieved. However, compared to the distributed estimation, distributed consensus requires  {\it $r$-robustness} which a stronger notion of network connectivity as discussed in the previous section. The following theorem provides a sufficient condition for the {iteration} \eqref{eqn:partiasdal} to reach to a consensus despite of the actions of faulty vehicles in the network. Before that, we present the definition of $f$-local set to ensure that the number of faulty vehicles in the network does not increase during the operation of the consensus dynamics. 
\begin{definition}[\textbf{$f$-local set}]
A set $\mathcal{S}\subset \mathcal{V}$ is $f$-local if it
contains at most $f$ nodes in the neighborhood of the other
nodes for all $t$, i.e., $|\mathcal{N}_i[t]\cap \mathcal{S}| \leq  f$, $\forall i \in \mathcal{V}\setminus \mathcal{S}, \forall t \in \mathbb{Z}_{\geq 0},
f \in \mathbb{Z}_{\geq 0}$.
\end{definition}

\begin{theorem}[\cite{Leblank}]
Suppose faulty vehicles form an $f$-local set. Then resilient asymptotic consensus is reached under the W-MSR iteration if the network is $(2f + 1)$-robust.
\label{thm:shreyashogantheoremmaster}
\end{theorem}
  It should be noted that the number of faulty nodes can be more than $f$ while they still form an $f$-local set. Thus, it  provides more freedom in the number possible faulty vehicles in Theorem \ref{thm:shreyashogantheoremmaster}. Lemma \ref{lem:rocbju} and Theorem \ref{thm:shreyashogantheoremmaster}  present the following theorem to show the ability of $\mathcal{P}(n,k)$ to perform robust distributed consensus. 
\begin{theorem}
Suppose the faulty vehicles form an $\lfloor \frac{k-1}{2}\rfloor$-local set in a $k$-nearest neighbor platoon. Then resilient asymptotic consensus on $\mathcal{P}(n,k)$ is reached under W-MSR dynamics, despite the action of faulty vehicles.
\label{thm:shreyashghgsmmaster}
\end{theorem}
Fig. \ref{fig:sstss54dsg7vnj} confirms the connectivity condition proposed by Theorem \ref{thm:shreyashghgsmmaster} for distributed consensus in the presence of faulty vehicles. Here, there exists a single faulty vehicle in the network (whose state is shown with red dashed line) and it is shown that $\mathcal{P}(10,3)$ is robust enough to overcome the action of the faulty vehicle.

\begin{figure}[t!] 
\centering
\includegraphics[scale=.4]{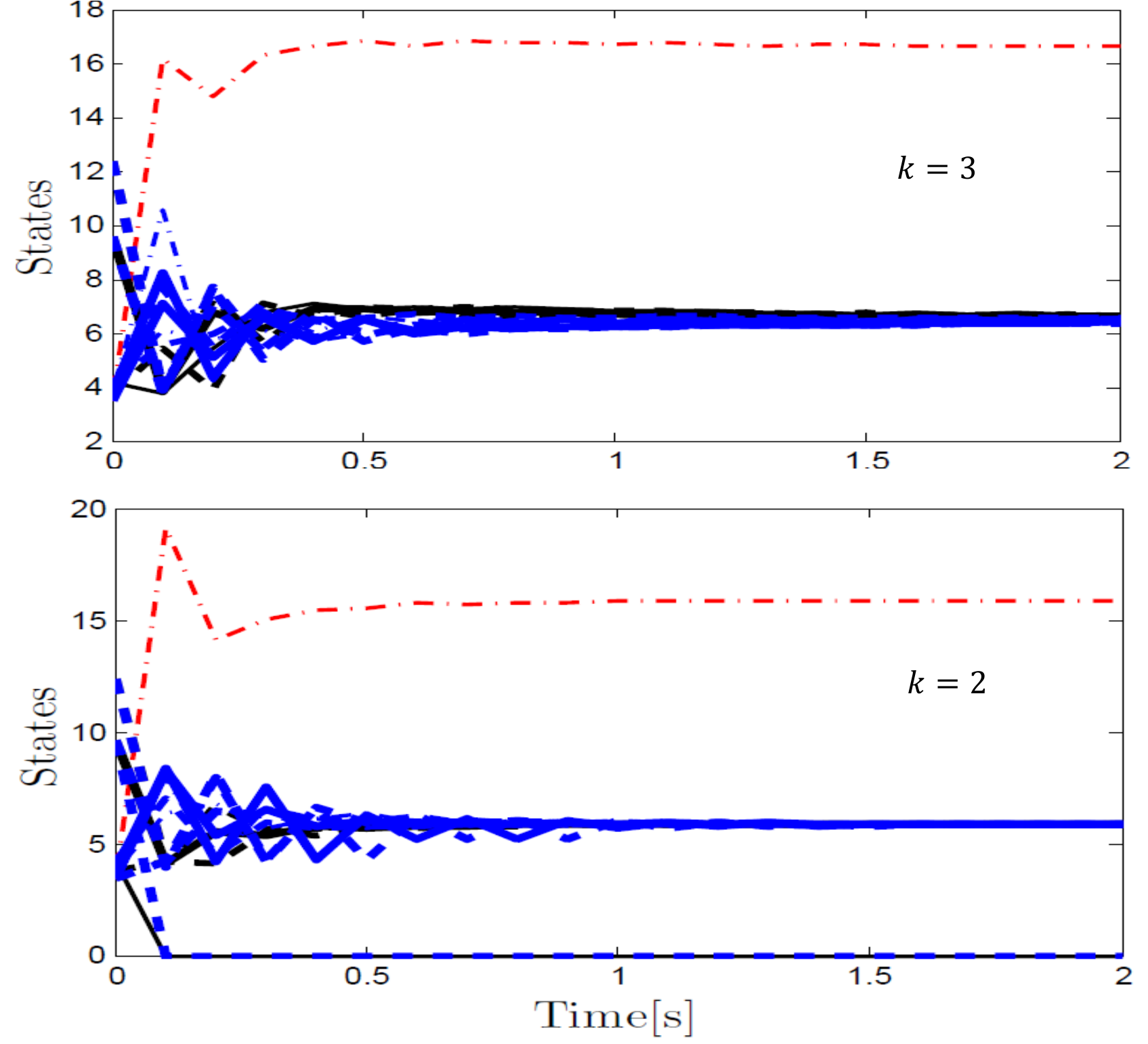}
\caption{Distributed consensus in the presence of a single faulty vehicle (red dashed line) for $\mathcal{P}(10,2)$ (bottom) and $\mathcal{P}(10,3)$ (top).}
\label{fig:sstss54dsg7vnj}
\end{figure}

\subsection{Network Formation in the Presence of Communication Disturbances}

The vehicle network formation is the third problem analyzed in this paper. {Let $p_i$ and $u_i$ denote the position and longitudinal velocity of vehicle $v_i$.} The objective is for each vehicle to maintain specific distances from its neighbors.  The desired vehicle formation will be formed by a specific constant distance $\Delta_{ij}$ between vehicles $v_i$ and $v_j$, which should satisfy $\Delta_{ij}=\Delta_{ik}+\Delta_{kj}$ for every triple $\{v_i,v_j,v_k\}  \subset\mathcal{V}$.  Considering the fact that each vehicle $v_i$ has access to its own position,  the positions of its neighbors,  and the desired inter-vehicular distances $\Delta_{ij}$, the control law for  vehicle $v_i$ is \cite{Hao}
\begin{align}
\ddot{p}_i(t)=\sum_{j\in \mathcal{N}_i}k_p \left (p_j(t)-p_i(t)+\Delta_{ij} \right )\nonumber\\
+k_u \left (u_j(t)-u_i(t)\right)+w_i(t),
\label{eqn:singlle}
\end{align}
where $k_p,k_u>0$ are control gains and $w_i(t)$ models communication disturbances.
Dynamics \eqref{eqn:singlle} in matrix form become
\begin{equation}
\dot{\boldsymbol{x}}(t)=\underbrace{\begin{bmatrix}
       \mathbf{0}_{n} & I_{n}         \\[0.3em]
     -k_pL & -k_uL
     \end{bmatrix}}_{A}{\boldsymbol{x}}(t)+\underbrace{\begin{bmatrix}
       \mathbf{0}_{n\times 1}        \\[0.3em]
    k_p\mathbf{\Delta}   
     \end{bmatrix}}_{B}+\underbrace{\begin{bmatrix}
       \mathbf{0}_{n}        \\[0.3em]
    I   
     \end{bmatrix}}_{F}\mathbf{w}(t),
\label{eqn:doub}
\end{equation}
 where $\boldsymbol{x}=[\mathbf{P} \hspace{2mm} \dot{\mathbf{P}}]^{\sf T}=[{{p}}_1, {{p}}_2, ..., {{p}}_n, \dot{{{p}}}_1, \dot{{{p}}}_2,  ..., \dot{{{p}}}_n]^{\sf T} $, $\mathbf{\Delta}=[\Delta_1, \Delta_2, ..., \Delta_n]^{\sf T}$ in which $\Delta_i=\sum_{j\in \mathcal{N}_i}\Delta_{ij}$. Here $\mathbf{w}(t)$ is the vector of disturbances. We want to quantify the effect of the communication disturbances on the inter-vehicular distances. For this, we need to define an appropriate performance measurement.  One such choice is
\begin{equation}
   \mathbf{y} =\mathcal{B}^{\sf T}{\mathbf{P}},
   \label{eqn:alternateoutput}
\end{equation}
where $\mathcal{B} \in \mathbb{R}^{n \times |\mathcal{E}|}$ is the incidence matrix associated with the network and $\mathbf{P}=[p_1, p_2, ..., p_n]^{\sf T}$ is the vector of positions. In this case we have an output associated with each connection, i.e., $y_{ij}={{p}}_i-{{p}}_j$ which is the distance between $v_i$ and $v_j$ at each time. With such performance output, we can quantify the sensitivity of inter-vehicular distances to communication disturbances. This sensitivity can be captured by an appropriate system norm from the disturbance signal to the desired output measurement. Here the system $\mathcal{H}_{\infty}$  norm is used which represents the worst case amplification of the disturbances over all frequencies and is widely used in the robustness analysis of vehicle platoons \cite{herman2015nonzero}. Such effect is discussed more formally in Theorem \ref{thm:sqrt23} which is proved in the Appendix.

\begin{theorem}
The system  $\mathcal{H}_{\infty}$ norm of \eqref{eqn:doub} from the external disturbances $\mathbf{w}(t)$ to $\mathbf{y} =\mathcal{B}^{\sf T}\mathbf{P}$ is 
\begin{align}
||G||_{\infty}=\begin{cases}
  \frac{2}{k_u\lambda_2\sqrt{4k_p-k_u^2\lambda_2}}, & \text{if }  \frac{\lambda_2k_u^2}{2k_p}\leq 1,\\
   \frac{1}{k_p\lambda_2^{\frac{1}{2}}} & \quad \text{otherwise}.\\
  \end{cases}
  \label{eqn:casess}
\end{align}
\label{thm:sqrt23}
\end{theorem}

Based on the above theorem, the algebraic connectivity of the network, $\lambda_2$, plays a major role in the $\mathcal{H}_{\infty}$ performance of the system. Hence, beside network connectivity and network robustness metrics mentioned in the previous sections, the algebraic connectivity is the third connectivity metric we consider for the performance of $k$-nearest neighbor platoons. One can  combine Proposition \ref{prop:algconn} and Theorem~\ref{thm:sqrt23} to find explicit graph-theoretic bounds on the system $\mathcal{H}_{\infty}$ norm of \eqref{eqn:doub}, as shown in Fig. \ref{fig:sstssddvnj}.

\begin{figure}[t!]
\centering
\includegraphics[scale=.39]{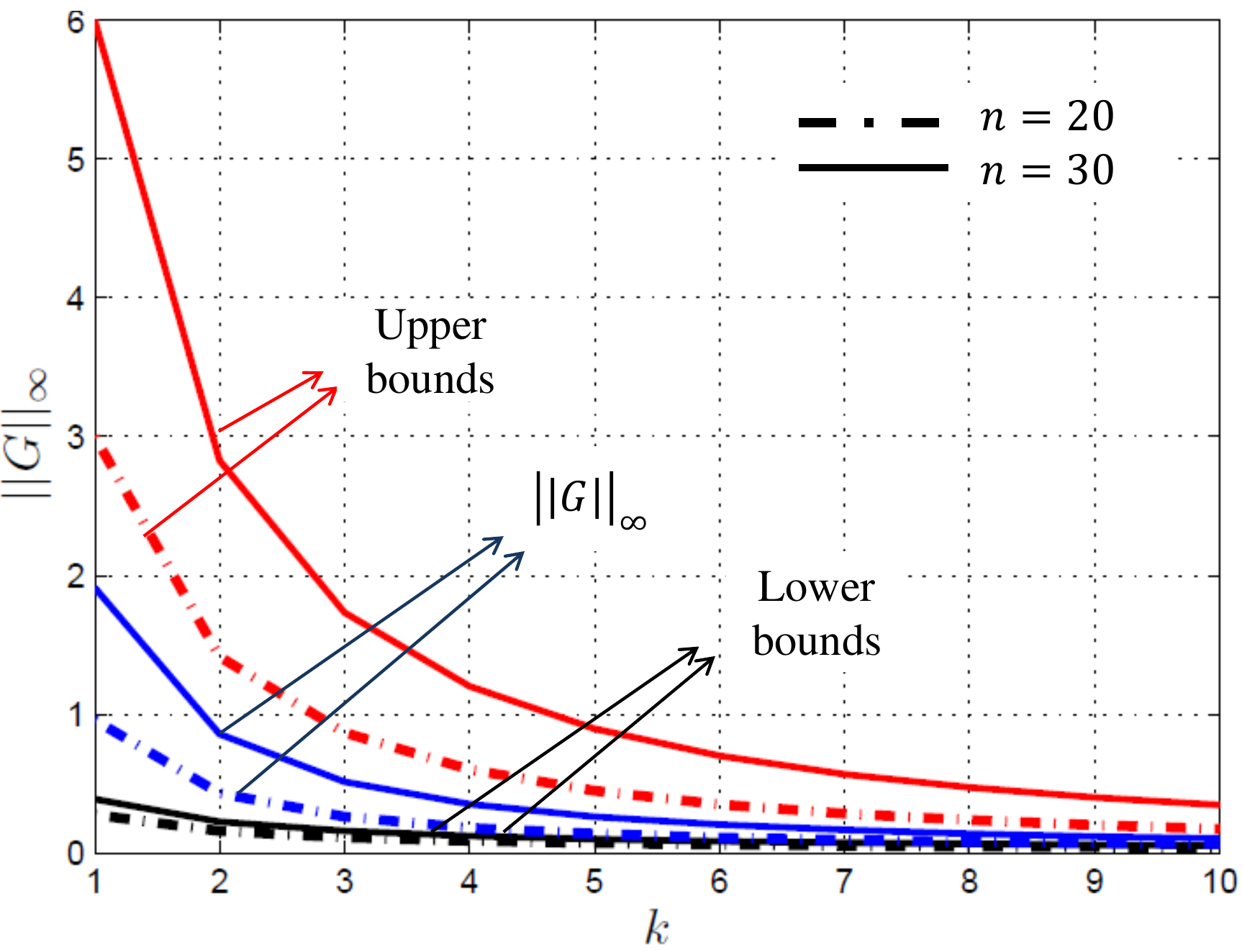}
\caption{Dependence of $\mathcal{H}_{\infty}$ norm of \eqref{eqn:doub} on network size $n$ and connectivity $k$ and bounds \eqref{eqn:twolowerbounds} on the algebraic connectivity.}
\label{fig:sstssddvnj}
\end{figure}

 This figure shows how the network local connectivity, captured by $k$, and the network network size $n$ have opposite effects on the system $\mathcal{H}_{\infty}$ norm of \eqref{eqn:doub}. Moreover, the upper and lower bounds can provide easily commutable necessary and sufficient conditions for having system $\mathcal{H}_{\infty}$ norm less than a certain number, instead of directly calculating the algebraic connectivity and \eqref{eqn:casess}. Fig. \ref{fig:sstss25ddvnj} shows the considerable  effect of increasing the connectivity index $k$ on the $\mathcal{H}_{\infty}$ performance of dynamics \eqref{eqn:doub} with parameters $k_p=5$ and $k_u=10$. According to this figure, for the traditional 1-nearest neighbor platoon with size $n=20$, the  $\mathcal{H}_{\infty}$ norm is about 2, while it drops below 1 for $k=2$ and below 0.5 for $k=4$. This shows a quadratic effect of increasing the connectivity index $k$ on the system $\mathcal{H}_{\infty}$ performance, as predicted by bounds in \eqref{eqn:twolowerbounds}.
 
 \begin{figure}[t!]
\centering
\includegraphics[scale=.42]{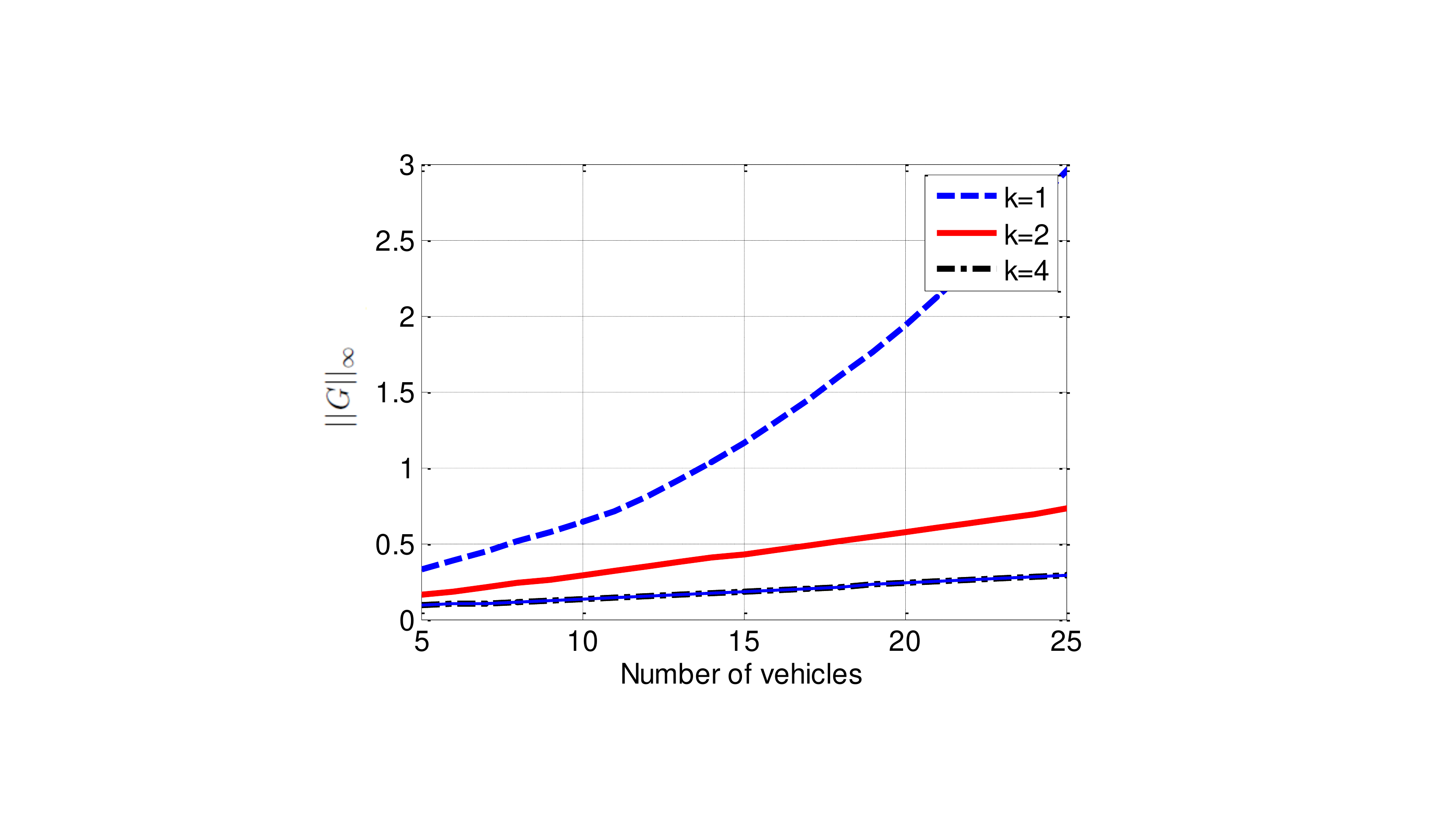}
\caption{The effect of increasing connectivity $k$ on the  $\mathcal{H}_{\infty}$ performance of dynamics \eqref{eqn:doub}.}
\label{fig:sstss25ddvnj}
\end{figure}

\subsection{Effect of the Network Scaling}

This subsection introduces the effect of the network scaling on the performance of each of the three distributed estimation and control algorithms on $\mathcal{P}(n,k)$ discussed in this section.

\textbf{Robust Distributed Estimation:} $k$-nearest neighbor platoons are secure networks in performing distributed estimation algorithms, as Theorem \ref{thm:robdeisdst} shows that it is possible to perform distributed estimation in the presence of up to $\lfloor \frac{k-1}{2} \rfloor$ faulty vehicles. Such a robustness metric  depends only on parameter $k$ and it is independent of the network size $n$. Hence, the performance of the distributed estimation algorithm is  only the function of local interactions of each agent.

\textbf{Robust Consensus:}  $k$-nearest neighbor platoons are also secure networks for distributed consensus, as Theorem \ref{thm:shreyashogantheoremmaster} indicates that they can tolerate up to $\lfloor \frac{k-1}{2} \rfloor$ faulty vehicles. Similar to distributed estimation case, the performance of the robust consensus algorithm is independent of the network size and managing the local interactions, i.e., parameter $k$, is sufficient to yield the desired performance.

\textbf{$\mathcal{H}_{\infty}$ Robustness to Disturbances:} According to Proposition \ref{prop:algconn}, the algebraic connectivity of large scale $k$-nearest neighbor platoons is not large, i.e., they are not good expanders. Hence, based on Theorem \ref{thm:sqrt23} and what is shown in Fig. \ref{fig:sstssddvnj}, the ability of these networks to mitigate the effect of disturbances, deceases as the size of the network increases. Unlike the two previous robustness metrics, the system $\mathcal{H}_{\infty}$  norm for fixed $k$ depends on the network size. Thus,  in order to make $\mathcal{P}(n,k)$ robust to external disturbances, it is required to adapt the number of local interactions $k$ with the network size $n$.


\section{Summary and Conclusions}
\label{sec:conclusion}
This paper investigates some network connectivity measures of $k$-nearest neighbor platoons. Explicit expressions, or graph-theoretic bounds, for each connectivity measure was proposed. Then these values were applied to provide graph-theoretic  conditions for three distributed estimation control algorithms and the results where demonstrated via simulations. It is shown that extending the traditional platooning topologies (which were based on interacting with the nearest neighbor) to $k$-nearest neighbors increases the resilience of distributed estimation and control algorithms to communication failures and external disturbances. A potential future direction is to analyze the performance of $k$-nearest neighbor platoons in more complicated maneuvers, e.g., traffic merging, and investigate the performance of the estimation and control algorithms for this particular topology of vehicle network.  

\bibliographystyle{IEEEtran}
\bibliography{refs}

\begin{appendix}

\begin{center}
Proof of Theorem \ref{thm:sqrt23}
\end{center}

\begin{proof}
First we show that the system $\mathcal{H}_{\infty}$ norms of \eqref{eqn:doub} from disturbance signals $\mathbf{w}(t)$ to performance outputs  $\mathbf{y} =\mathcal{B}^{\sf T}{\mathbf{P}}$ and  $\mathbf{y} =L^{\frac{1}{2}}{\mathbf{P}}$ are the same. For the output measurement $\mathbf{y}=\mathcal{B}^{\sf T}{\mathbf{P}}$ we have $G^*G=F^{\sf T}(s^*I-A)^{-\sf T}\mathcal{B}\mathcal{B}^{\sf T}(sI-A)^{-1}F=F^{\sf T}(s^*I-A)^{-\sf T}L(sI-A)^{-1}F$ and as system $\mathcal{H}_{\infty}$ norm is a function of the spectrum of $G^*G$, identical results will be obtained as if one used $\mathbf{y} =L^{\frac{1}{2}}{\mathbf{P}}$ instead of $\mathbf{y} =\mathcal{B}^{\sf T}{\mathbf{P}}$. Hence, it is sufficient to find the system $\mathcal{H}_{\infty}$ norm of \eqref{eqn:doub} from disturbances to $\mathbf{y} =L^{\frac{1}{2}}{\mathbf{P}}$. Let $\Lambda = V^{\sf T}LV$ be the eigendecomposition of $L$, where $V$ may be taken to be orthogonal. Consider the invertible change of states $\tilde{\boldsymbol{x}} = (V^{\sf T}\boldsymbol{x},V^{\sf T}\dot{\boldsymbol{x}})$. Then a straightforward computation shows that
\begin{equation}\label{Eq:SwingTransformed1}
\begin{aligned}
\dot{\tilde{\boldsymbol{x}}} &= \begin{bmatrix}
0 & I_n\\ -k_p\Lambda & -k_u\Lambda
\end{bmatrix}\tilde{\boldsymbol{x}} + \begin{bmatrix}
0 \\ V^{\sf T}
\end{bmatrix}w\\
y &= \begin{bmatrix}
L^{\frac{1}{2}}V & 0
\end{bmatrix}\tilde{\boldsymbol{x}}\,.
\end{aligned}
\end{equation}
The model \eqref{Eq:SwingTransformed1} has the same transfer function as \eqref{eqn:doub}, and hence the same system norm. Now consider an input/output transformation on \eqref{Eq:SwingTransformed1}, where $\bar{y} = V^{\sf T}y$ and $\bar{w} = V^{\sf T}w$\,, knowing the fact that such input/output transformation preserves the system $\mathcal{H}_{\infty}$ norm \cite{PiraniJohnCDC}. Hence, the transformed system
\begin{equation}\label{Eq:SwingTransformed2}
\begin{aligned}
\dot{\tilde{\boldsymbol{x}}} &= \begin{bmatrix}
0 & I_n\\ -k_p\Lambda & -k_u\Lambda
\end{bmatrix}\tilde{\boldsymbol{x}} + \begin{bmatrix}
0 \\ \underbrace{V^{\sf T}V}_{= I_n}
\end{bmatrix}\bar{w}\\
\bar{y} &= 
\underbrace{
\begin{bmatrix}
V^{\sf T}L^{\frac{1}{2}}V & 0
\end{bmatrix}}_{=\begin{bmatrix}\Lambda^{\frac{1}{2}} & 0 \end{bmatrix}}
\tilde{\boldsymbol{x}}\,.
\end{aligned}
\end{equation}
has the same system norm as \eqref{Eq:SwingTransformed1}. The system \eqref{Eq:SwingTransformed2} is comprised of $n$ decoupled subsystems, each of the form
\begin{equation}\label{Eq:SwingTransformed3}
\begin{aligned}
\dot{\tilde{\boldsymbol{x}}}_i &= \begin{bmatrix}
0 & 1\\ -k_p\lambda_i & -k_u\lambda_i
\end{bmatrix}\tilde{\boldsymbol{x}}_i + \begin{bmatrix}
0 \\ 1
\end{bmatrix}\bar{w}_i\\
\bar{y}_i &= 
\begin{bmatrix}
\lambda_i^{\frac{1}{2}} & 0
\end{bmatrix}
\tilde{\boldsymbol{x}}_i\,.
\end{aligned}
\end{equation}
with transfer functions
$$
\tilde{G}_i(s) = \frac{\lambda_i^{\frac{1}{2}}}{s^2 + k_u\lambda_is + k_p\lambda_i}\,, \qquad i \in \{1,\ldots,n\}\,.
$$
which gives $\tilde{G}_1(s) = 0$. For $i \in \{2,\ldots,n\}$,  we have 
\begin{align*}
|\tilde{G}_i(j\omega)|^2&=\tilde{G}_i(-j\omega)\tilde{G}_i(j\omega)=\frac{\lambda_i}{\underbrace{(k_p\lambda_i-\omega^2)^2+k_u^2\lambda_i^2\omega^2}_{f(\omega)}}.
\end{align*}
Maximizing $|\tilde{G}_i(j\omega)|^2$ with respect to $\omega$ is equivalent to minimizing $f(\omega)$. By setting $\frac{df(\omega)}{d\omega}=0$ we get $\bar{\omega}_1=0$ and $\bar{\omega}_2=({k_p\lambda_i}-\frac{1}{2}k_u^2\lambda_i^2)^{\frac{1}{2}}$ as critical points. Here 
$\bar{\omega}_2$ is the global minimizer of $f(\omega)$, unless $\frac{k_u^2\lambda_i}{2k_p}>1$. Substituting these critical values back into the formula for $|\tilde{G}_i(j\omega)|^2$, we find for $i \in \{2,\ldots,n\}$ that
\begin{align}
||\tilde{G}_i||_{\infty}= \begin{cases}
  \frac{2}{k_u\lambda_i\sqrt{4k_p-k_u^2\lambda_i}},  &  \text{if } \frac{\lambda_ik_u^2}{2k_p}\leq 1,\\
   \frac{1}{k_p\lambda_i^{\frac{1}{2}}} &  \text{otherwise}.\\
  \end{cases}
  \label{eqn:cascess}
\end{align}
Since $0 < \lambda_2 \leq \lambda_3 \leq \cdots \leq \lambda_n$  and $||\tilde{G}_i||_{\infty}$ is a monotonically decreasing function of $\lambda_i$,  the result follows.
\end{proof}

\end{appendix}

\end{document}